\documentclass[twoside, centertags]{amsart}

\usepackage[cmtip,color,all]{xy}
\usepackage{amsmath}
\usepackage{amsfonts}
\usepackage{amssymb}
\usepackage{amsthm}
\usepackage{graphicx}
\usepackage{mathrsfs}
\usepackage{hyperref}

\newtheorem{theorem}{Theorem}[section]
\newtheorem{corollary}[theorem]{Corollary}

\newtheorem{proposition}[theorem]{Proposition}
\theoremstyle{definition}
\newtheorem{definition}[theorem]{Definition}
\theoremstyle{definition}
\newtheorem{remark}[theorem]{Remark}
\theoremstyle{definition}
\newtheorem{example}[theorem]{Example}

\newcommand{\C}{\mathscr{C}}

\newcommand{\M}{\mathscr{M}}
\newcommand{\Tr}{\mathrm{Tr}}
\newcommand{\Wh}{\mathrm{Wh^{Diff}}}
\newcommand{\tr}{\mathrm{tr}}

\DeclareMathOperator*{\hocolim}{hocolim}

\DeclareMathOperator*{\holim}{holim}
\DeclareMathOperator{\map}{map}

\DeclareMathOperator{\pr}{pr}

\DeclareMathOperator{\simp}{simp}

\newcommand{\ZZ}{\mathbb{Z}}

\newcommand{\Rfd}{\mathcal{R}_{hf}}

\newcommand{\E}{\mathbb{E}}

\begin{document}

\title{On transfer maps in the algebraic $K$-theory of spaces} % 
\author[G. Raptis]{George Raptis}

\begin{abstract}
We show that the Waldhausen trace map $\Tr_X \colon A(X) \to QX_+$, which defines a natural splitting map from the 
algebraic $K$-theory of spaces 
to stable homotopy, is natural up to \emph{weak} homotopy with respect to transfer maps in algebraic $K$-theory and 
Becker-Gottlieb transfer maps respectively.
\end{abstract}

\address{\newline
G. Raptis \newline
Fakult\"at f\"ur Mathematik, Universit\"at Regensburg, 93040 Regensburg, Germany}
\email{georgios.raptis@ur.de}

\maketitle

\section{Introduction}

This paper addresses a question about the properties of transfer maps in the algebraic $K$-theory of spaces. In order to motivate and state this question, we first recall a few facts about the algebraic $K$-theory of spaces.

\medskip 

\noindent The algebraic $K$-theory of spaces, introduced by Waldhausen (see \cite{Wa}), is a functor from the category spaces which takes values in infinite loop spaces. The algebraic $K$-theory $A(X)$ of the space $X$, also known as the $A$-theory of $X$, encodes important invariants of $X$ that are of great interest in both homotopy theory and geometric topology. $A(X)$ 
is defined as the Waldhausen $K$-theory of a category of retractive spaces over $X$ with certain finiteness properties, or equivalently, when $X$ is based and path-connected, as the Waldhausen 
$K$-theory of a category of modules over the ring spectrum $\Sigma^{\infty} \Omega X_+$ that satisfy analogous finiteness assumptions. There is a natural homotopy fiber sequence:
\begin{equation} \label{basic-fib-seq}
\xymatrix{
QX_+ \ar[r]^{\eta_X} & A(X) \ar[r] & \Wh(X) \\
}                                                            
\end{equation}
where $QX_+ \colon = \Omega^{\infty} \Sigma^{\infty} X_+$ and $\eta_X$ denotes the unit transformation. If $X$ is a compact smooth manifold, the space $\Omega \Wh(X)$ is homotopy equivalent 
to the stable parametrized $h$-cobordism space of $X$ \cite{WJR}. The unit transformation admits a natural retraction up to homotopy, given by the Waldhausen trace map,
$$\Tr_X \colon A(X) \to QX_+,$$
and therefore we obtain from \eqref{basic-fib-seq} a natural homotopy equivalence of infinite loop spaces:
\begin{equation} \label{prod-dec}
A(X) \simeq QX_+ \times \Wh(X).
\end{equation}

\noindent In addition to the covariant functoriality, $A$-theory also admits transfer maps for fibrations $p \colon E \to B$ with homotopy finite fibers. In terms of the definitions indicated above, the $A$-theory transfer map 
$$p^* \colon A(B) \to A(E)$$
is given either by pulling back a retractive space over $B$ to a retractive space over $E$, or by restriction along the associated map of ring spectra $\Sigma^{\infty} \Omega E_+ \to \Sigma^{\infty} \Omega B_+$
when $p$ is a based map between based path-connected spaces. The point here is that either one of these two operations preserves the required finiteness properties when $p$ has homotopy finite fibers. Similarly to 
other types of transfer or `wrong-way' maps, these transfer maps enrich the structure of $A$-theory and their study is closely related to index theorems in $A$-theory (see, for example, \cite{Wi}). 

\medskip

The question that we address in this paper is concerned with understanding the $A$-theory transfer maps in terms of the product decomposition of 
$A$-theory in \eqref{prod-dec}. Here follows a 
brief summary of some known results in this direction. Let $p \colon E \to B$ denote a fibration with homotopy finite fibers. 
\begin{itemize}
 \item[(A)] The composition of infinite loop space maps
 $$QB_+ \xrightarrow{\eta_B} A(B) \xrightarrow{p^*} A(E) \xrightarrow{\Tr_E} QE_+$$ 
 is naturally homotopic to the Becker-Gottlieb transfer map $QB_+ \xrightarrow{\tr_p} QE_+$ associated to $p$ (Lind--Malkiewich \cite{LM}; see also \cite{DJ, D}).
 \item[(B)] The composition $QB_+ \xrightarrow{\eta_B} A(B) \xrightarrow{p^*} A(E)$ does not factor in general through the unit map $\eta_E \colon QE_+ \to A(E)$. Moreover, the composition 
 $$QB_+ \xrightarrow{\eta_B} A(B) \xrightarrow{p^*} A(E) \rightarrow \Wh(E)$$ 
 encodes higher Whitehead torsion invariants of $p$. According to the Dwyer-Weiss-Williams smooth index theorem \cite{DWW, Wi, RS2}, if $p$ is fiberwise equivalent to a fiber bundle of 
 compact smooth manifolds, then the square 
 $$
 \xymatrix{
 QB_+ \ar[r]^{\eta_B} \ar[d]_{\tr_p} & A(B) \ar[d]_{p^*} \\
 QE_+ \ar[r]^{\eta_E} & A(E) 
 }
 $$
commutes up to a canonical homotopy.  
\end{itemize}
Although not immediately related to $A$-theory transfer maps, the following result will also be relevant:
\begin{itemize}
 \item[(C)] The Becker-Gottlieb transfer maps are functorial in the stable homotopy category: given fibrations $p \colon E \to B$ and $q \colon V \to E$ with homotopy finite fibers, then the 
 composition $QB_+ \xrightarrow{\tr_p} QE_+ \xrightarrow{\tr_q} QV_+$ is homotopic to $QB_+ \xrightarrow{\tr_{p\circ q}} QV_+$ (Klein--Malkiewich \cite{KM}).
\end{itemize}
We note that the problem of identifying a canonical and natural homotopy for the functoriality of the Becker-Gottlieb transfer maps in (C) remains open. 

\medskip

\noindent In (B), we see that the Dwyer-Weiss-Williams theorem explains the naturality (or lack thereof) of the unit transformation with respect to the Becker-Gottlieb transfer maps and $A$-theory transfer maps, respectively. Similarly, we consider in this paper the analogous question for the Waldhausen trace map and ask about the commutativity of the following square:
\begin{equation} \label{main-question} \tag{$\ast$}
\xymatrix{\ar@{}[dr] | {?}
A(B) \ar[d]_{p^*} \ar[r]^{\Tr_B} & QB_+ \ar[d]^{\tr_p} \\
A(E) \ar[r]_{\Tr_E} & QE_+.
 }
\end{equation}
With the exception of a suggestive comment in \cite[p. 394]{Wa4}, it seems that this question has not been considered in the literature until recently in connection with the results stated in (A)--(C) above. Note that the commutativity of \eqref{main-question} implies that the composition 
$$\Wh(B) \stackrel{\subset}{\to} A(B) \xrightarrow{p^*} A(E) \xrightarrow{\Tr_E} QE_+$$
is trivial. We point out that the commutativity of (\ref{main-question}) also easily implies the statements in (A) and (C). In addition, any naturality properties for a choice of a 
homotopy in \eqref{main-question} will imply analogous naturality properties for the identifications in (A) and (C). 

\medskip

We now state our main results in connection with this question \eqref{main-question} and give an outline of the proofs. Our main result (Theorem \ref{main-thm-2}) is that the square \eqref{main-question} commutes up to weak homotopy. 
(We recall that two maps $f, g \colon X \to Y$ are \emph{weakly homotopic} if for any map $h \colon K \to X$ where $K$ is homotopy finite, the maps $fh$ and $gh$ are homotopic. In this case, we also say that $f$ and $g$ agree up to weak homotopy.) In particular, this implies that the square \eqref{main-question} commutes on homotopy groups. For the proof of Theorem \ref{main-thm-2}, we work with the definition of $A$-theory in terms of group completion and the plus construction (see \cite{Wa, Wa3}). Assuming that $B$ is based and 
path-connected, let $\M^{n,k}_B$ denote a model for the classifying space of fibrations of retractive spaces over $B$ which are of type $(B \vee \bigvee_1^k S^n, B)$. It is known that $A(B)$ is obtained from these spaces by passing to $k \to \infty$ (addition of summands) and $n \to \infty$ (suspension) and applying the plus construction. Using this description of $A$-theory, we apply
a result of Hausmann--Vogel about acyclic maps \cite{HV} which allows us to find for any map $h \colon K \to A(B)$, where $K$ is homotopy finite and path-connected, a homotopy commutative diagram as follows:
$$
\xymatrix{
K' \ar[d]_j \ar[r] & \M_B^{n,k} \ar[d] \\
K \ar[r]_h & A(B)
}
$$
where $j$ is an acyclic map. This `factorization of $h$ up to an acyclic map', which may also be of independent interest, effectively reduces the proof of Theorem \ref{main-thm-2} to showing that the square \eqref{main-question} commutes after precomposition with the canonical map $\M_B^{n,k} \to A(B)$. Assuming that $B$ is homotopy finite, this claim is then a special case of Theorem \ref{main-thm-1}. For the 
proof of Theorem \ref{main-thm-1}, we consider a relative version of the $A$-theory characteristic of Dwyer-Weiss-Williams and use the theorem of Lind--Malkiewich \cite{LM} (stated in (A)) in order to express the composition 
\begin{equation} \label{main-map} 
\M^{n,k}_B \to A(B) \xrightarrow{\Tr_B} QB_+ 
\end{equation}
in terms of Becker-Gottlieb transfer maps (Corollary \ref{identify-trace-map}). Finally, using this identification of \eqref{main-map}, the proof of Theorem \ref{main-thm-1} is reduced to the functoriality of the Becker-Gottlieb transfer maps shown by Klein--Malkiewich \cite{KM} (stated in (C)).

\medskip

\noindent \textbf{Organization of the paper.} In Section 2, we begin with the definition of $A$-theory in terms of spherical objects and group completion which will be especially useful for our purposes. Then we briefly recall the definition 
of bivariant $A$-theory and the coassembly map in order to introduce a relative version of the $A$-theory characteristic and explore some of its properties. In addition, we consider the classifying space $\M_X(V)$ -- a generalization of 
$\M_X^{n,k}$ -- and identify the canonical map $\M_X(V) \to A(X)$ in terms of the $A$-theory characteristic. In Section 3, we first recall the definition and some properties of the $A$-theory transfer 
maps and state the theorems of Lind--Malkiewich \cite{LM} and Klein--Malkiewich \cite{KM}. Combining these results, we obtain an identification of the map \eqref{main-map} in terms of Becker-Gottlieb transfer maps (Corollary \ref{identify-trace-map}) and then prove Theorem \ref{main-thm-1}. In Section 4, 
we first review some facts about acyclic maps and the plus construction and then conclude with the proof of the main theorem (Theorem \ref{main-thm-2}). We end in Section 5 with a couple of final remarks about the commutativity of \eqref{main-question}. Firstly, we discuss the limitations of our approach for deciding whether 
\eqref{main-question} commutes up to a canonical homotopy. Secondly, we recall the natural factorization of 
the Waldhausen trace map through topological Hochschild homology (THH) and comment on the analogous question to 
\eqref{main-question} for THH. 

\medskip

\noindent \textbf{Acknowledgements.} I would like to thank John Lind, Cary Malkiewich and Mona Merling for many interesting discussions about the commutativity of \eqref{main-question} and their interest in this work. 
I also thank John Rognes and Wolfgang Steimle for helpful discussions and for their comments on earlier drafts of 
this work. I gratefully acknowledge the support and the hospitality of the Isaac Newton Institute during the Research Programme ``\emph{Homotopy harnessing higher structures}'' and the support from the \textit{SFB 1085 -- Higher Invariants} (University of Regensburg).

\section{The $A$-theory characteristic}

\subsection{Recollections} \label{recollect-1} Let $X$ be a space and let $\Rfd(X)$ denote the Waldhausen category of homotopy finite retractive spaces over $X$ (see \cite[2.1]{Wa}). 
The Waldhausen $K$-theory of $\Rfd(X)$, 
$$A(X) \colon = K(\Rfd(X)),$$
is the \emph{algebraic K-theory of X} ($A$-theory). We recall that $A$-theory defines a covariant functor which preserves weak homotopy equivalences and 
takes values in infinite loop spaces. We will generally denote by $f_* \colon A(X) \to A(X')$ the map in $A$-theory induced by a map $f \colon X \to X'$.

Suppose that $X$ is based and path-connected. For $n \geq 0$, let $\C^n_X$ denote the full Waldhausen subcategory of $\Rfd(X)$ which consists of the objects which are weakly equivalent in $\Rfd(X)$ to a retractive object 
of the form $(X \vee \bigvee_{1}^k S^n, X)$ for some $k \geq 0$. 
Let $w \C_X^n$ denote the subcategory of weak equivalences in $\C_X^n$. The category $w\C_X^n$ is a symmetric monoidal category with respect to the coproduct pairing in $\C_X^n$. As a consequence, the geometric realization of this category $|w \C_X^n|$ carries the structure of an $\E_{\infty}$-space. 

The exact inclusion functor $\C_X^n \hookrightarrow \Rfd(X)$ induces a canonical map in Waldhausen $K$-theory
$K(\C_X^n) \rightarrow K(\Rfd(X)) = A(X)$. Note that the suspension functor on $\Rfd(X)$ restricts to exact functors $\Sigma \colon \C_X^n \to \C_X^{n+1}$, $n \geq 0$. After stabilizing with respect to the suspension functor, we obtain a canonical homotopy equivalence:
\begin{equation} \label{group-comp-A-theory}
\hocolim_{(\Sigma)} K(\C_X^n) \stackrel{\simeq}{\longrightarrow} \hocolim_{(\Sigma)} K(\Rfd(X)) \simeq A(X).
\end{equation}
This homotopy equivalence is essentially a consequence of Waldhausen's theorem on spherical objects \cite[1.7]{Wa}. Since the cofibrations in $\C_X^n$ split up to weak equivalence, it follows from \cite[1.8]{Wa} that the homotopy colimit of the Waldhausen $K$-theories $K(\C_X^n)$, $n \geq 0$, is homotopy equivalent to the homotopy colimit of the corresponding 
diagram of the group completions of $|w \C_X^n|$ for $n \geq 0$ (that is, the $K$-theory of the symmetric monoidal category $w \C_X^n$ for every $n \geq 0$). This identification of $A$-theory is a variation of a similar identification in terms of matrices which was shown in \cite[2.2]{Wa}. 

\subsection{The coassembly map} Let $p \colon E \to B$ be a fibration and let $\Rfd(p)$ denote the Waldhausen category of retractive spaces $(V, E)$ over $E$ which restrict fiberwise for each $b \in B$ to a homotopy finite retractive space $(V_b, E_b) \in \Rfd(E_b)$,
$$
\xymatrix{
V_b \ar[r] \ar@/_/[d]_{r_b} & V \ar@/_/[d]_r \ar@/^/[rd]^{q = p r} & \\
E_b \ar[r] \ar@/_/[u]_{i_b} & E \ar[r]_p \ar@/_/[u]_{i} & B 
}
$$
and $q \colon V \to B$ is a fibration. We recall that the assignment $p \mapsto \Rfd(p)$ is covariant in $E$ (as a space over $B$), contravariant in $B$, and it is homotopy invariant in both variances (see \cite[Section 3]{RS1} for more 
details). The Waldhausen $K$-theory of $\Rfd(p)$,
$$A(p) \colon = K(\Rfd(p)),$$ is the \emph{bivariant A-theory of p}. Following \cite[4.1]{RS1}, and assuming without essential loss of generality that $B = |B_{\bullet}|$ is the geometric realization of a simplicial set, there is a coassembly functor:
\begin{equation} \label{coassembly-functor} 
w \Rfd(p) \longrightarrow \holim_{\simp(B)} w \Rfd(E_{|?})
\end{equation}
and a coassembly map in $A$-theory:
\begin{equation} \label{coassembly-map}
\nabla_p \colon A(p) \longrightarrow \holim_{\simp(B)} A(E_{|?})
\end{equation}
where $\simp(B)$ is the simplex category of $B_{\bullet}$ and $E_{|\sigma}$ denotes the total space of the restriction of $p$ along the inclusion of the simplex $\sigma$. The coassembly map 
\eqref{coassembly-map} is obtained from \eqref{coassembly-functor} by passing to the $S_{\bullet}$-construction in \eqref{coassembly-functor} and then permuting a homotopy limit out of a geometric realization (`homotopy 
limit problem'). The coassembly map is a map of infinite loop spaces which is natural with respect to $p$. The target of the coassembly map can be naturally identified up to homotopy equivalence with the space of 
sections of a fibration $A_B(p) \colon A_B(E) \to B$ which is obtained from 
$p$ by applying the $A$-theory functor fiberwise. With this identification in mind, the coassembly map can be roughly expressed on the objects of $\Rfd(p)$ as follows:
\begin{equation} \label{heuristic} 
\nabla_p \colon \big((V, E) \in A(p)\big) \mapsto \big(b \in B \mapsto (V_b, E_b) \in A(E_b)\big).
\end{equation}
(We refer to \cite[Section 4]{RS1} and \cite{Wi} for more details.)
The inclusion maps $E_{|\sigma} \to E$, $\sigma \in \simp(B)$, induce a collection of compatible maps $A(E_{|?}) \to A(E)$, and therefore also a canonical map 
$$A_B(E) \simeq \hocolim_{\simp(B)} A(E_{|?}) \to A(E),$$
from which we obtain a `forgetful' map:
$$\holim_{\simp(B)} A(E_{|?}) \to \map(B, A(E)).$$ 

\begin{definition}
Let $p \colon E \to B = |B_{\bullet}|$ be a fibration and $(V, E) \in \Rfd(p)$. The \emph{(parametrized) $A$-theory characteristic} of $(V, E)$ is the map 
$$\chi_B(V,E) \colon = (\nabla_p(V, E) \colon B \longrightarrow A_B(E)).$$
We also denote the \emph{(unparametrized) A-theory characteristic} of $(V,E)$ by:
$$\chi(V, E) \colon B \xrightarrow{\chi_B(V,E)} A_B(E) \to A(E).$$
\end{definition}

\begin{remark}
Using the heuristic description of the coassembly map in \eqref{heuristic}, the parametrized $A$-theory characteristic of $(V,E)$ can be described as 
$$\chi_B(V,E) \colon (b \in B) \mapsto ((V_b, E_b) \in A(E_b)).$$ 
Similarly, the unparametrized $A$-theory characteristic of $(V,E)$ can be described as 
$$\chi(V,E) \colon (b \in B) \mapsto ((E \cup_{E_b} V_b, E) \in A(E)).$$
\end{remark}

\begin{example}
Let $p \colon E \to B = |B_{\bullet}|$ be a fibration with homotopy finite fibers. Then $(V = E \times S^0, E)$ is an object of $\Rfd(p)$. 
The parametrized $A$-theory characteristic of $(E \times S^0, E)$ agrees with the A-theory characteristic of $p$ as defined by Dwyer--Weiss--Williams \cite{DWW} (see also \cite[4.2]{RS1}, \cite{Wi}). This will be denoted here by 
$\chi_B(p) \colon B \to A_B(E)$ and can be described as $(b \in B) \mapsto ((E_b \times S^0, E_b) \in A(E_b))$. We will also denote the associated unparametrized $A$-theory characteristic by
$$\chi(p) \colon = \chi(E \times S^0, E).$$   
Note that the path component of $(X \times S^0, X) \in A(X)$, where $X$ is homotopy finite and path-connected, is identified with the classical Euler characteristic $\chi(X) \in \ZZ \cong \pi_0 A(X)$.
Thus, we may regard $\chi_B(p)$ as a parametrized $A$-theoretic refinement of the classical Euler characteristic. If $p \colon B \times F \to B$ is a trivial fibration where $F$ is homotopy finite, then
the section $\chi_B(p) \colon B \to A_B(B \times F) \simeq B \times A(F)$ is given up to canonical homotopy by the constant map $B \to A(F)$ at $(F \times S^0, F) \in A(F)$. 
\end{example}

\begin{proposition} \label{relative-characteristic}
Let $p \colon E \to B = |B_{\bullet}|$ be a fibration with homotopy finite fibers and let $(V, E)$ be an object of $\Rfd(p)$, depicted as follows:
$$
\xymatrix{
V \ar@/^/[rd]^{q = p r} \ar@/^/[d]^r & \\
E \ar[r]_p \ar@/^/[u]^i & B. 
}
$$
Then the composite map $r_* \chi(q) \colon B \xrightarrow{\chi(q)} A(V) \xrightarrow{r_*} A(E)$ is homotopic to the map 
$\chi(p) + \chi(V, E) \colon B \to A(E)$, naturally in $p$. 
\end{proposition}
\begin{proof}
Note that $q$ also has homotopy finite fibers, so $\chi(q)$ is well-defined. By the naturality of the coassembly map \eqref{coassembly-map} with respect to fiberwise maps over $B$, it follows that the composite map 
$$r_* \chi(q) \colon B \xrightarrow{\chi(q)} A(V) \xrightarrow{r_*} A(E)$$ 
agrees with the $A$-theory characteristic of $r_*(V \times S^0, V) = (V \sqcup E, E) \in \Rfd(p)$. Then the required identification of homotopy classes follows from the canonical 
cofiber sequence in $\Rfd(p)$: 
$$(E \times S^0, E) \rightarrowtail (V \sqcup E, E) \twoheadrightarrow (V, E).$$ 
The naturality statement expresses the fact that this homotopic 
identification of $\chi(V,E)$ as $r_* \chi(q) - \chi(p)$ is natural in $(V,E)$ with respect to the covariant and contravariant functorialities of $p \mapsto A(p)$. 
\end{proof}

\subsection{The classifiyng spaces $\M_X(V)$} \label{classifying-1} Let $X$ be a space and let $(V, X)$ be an object in $\Rfd(X)$ .  Let $\C_X(V)$ denote the full subcategory of $\Rfd(X)$ whose objects are the retractive spaces which are weakly equivalent to $(V, X)$. The classifying space of the subcategory of weak equivalences $w \C_X(V)$,
$$\M_X(V) \colon = |w \C_X(V)|,$$
is a classifying space for fibrations of retractive spaces of type $(V,X)$. In other words, we have a homotopy equivalence:
$$\M_X(V) \simeq  \mathrm{B Aut^h_X}(V)$$
where $\mathrm{Aut^h_X}(-)$ denotes the derived mapping space of homotopy self-equivalences in $\Rfd(X)$ (see, for example, \cite[Proposition 2.3]{DK}). Note that $\M_X(V)$ is a 
connected component of $|w \Rfd(X)|$. There is universal fibration of retractive objects 
$$
\xymatrix{
\M_X(V) \times X  \ar@{>->}[r]_(.6)i \ar[rd]_{\pr} \ar@/^1pc/[rr]^{\mathrm{Id}} & \mathscr{E}_X(V) \ar[r]_(.4)r \ar[d]^{p_X^V} & \M_X(V) \times X \ar[dl]^{\pr} \\
& \M_X(V)
}
$$
whose fibers over $\M_X(V)$ are weakly equivalent to the retractive space $(V, X)$. Note that the diagram above defines an object in $\Rfd(\M_X(V) \times X \xrightarrow{\pr} \M_X(V))$. 
The inclusion functor $\C_X(V) \subset \Rfd(X)$ induces a canonical map 
$$j_X^V \colon \M_X(V) \to |w \Rfd(X)| \to A(X).$$

\begin{example}
Suppose that $X$ is based and path-connected and consider the retractive space $(V, X) = (X \vee \bigvee_1^k S^n, X) \in \Rfd(X)$ for some $n, k \geq 0$. In this case, we denote $\C_X(V)$ by $\C_X^{n,k}$ for simplicity. The classifying space of the subcategory of weak equivalences $w \C_X^{n,k}$, denoted 
$$\M_X^{n,k} \colon = |w \C_X^{n,k}|,$$
is a classifying space for fibrations of retractive spaces of type $(X \vee \bigvee_1^k S^n, X)$, and we have a homotopy equivalence $\M_X^{n,k} \simeq  \mathrm{B Aut^h_X}(X \vee \bigvee_1^k S^n).$
Note that $\C_X^{n,k}$ is a full subcategory of $\C_X^n$ and $\M_X^{n,k}$ is a connected component of $|w \C_X^n|$. 
Similarly, the inclusion functor $\C_X^{n,k} \subset \Rfd(X)$ induces a canonical map 
$$j_X^{n,k} \colon \M_X^{n,k}\to |w \Rfd(X)| \to A(X).$$
By considering these maps for all $k \geq 0$, passing to the group completions of $\M_X^n = \bigsqcup_{k \geq 0} \M_X^{n,k}$ for $n \geq 0$, and letting $n \to \infty$ 
(along the suspension functor), we obtain the homotopy equivalence that was discussed after \eqref{group-comp-A-theory}.  
\end{example}

Let $\pi \colon \M_X(V) \times X \to X$ denote the projection map.

\begin{proposition} \label{comparison-1}
The map $j_X^V \colon \M_X(V) \to A(X)$ is naturally homotopic to the composite map
$$\M_X(V) \xrightarrow{\chi(\mathscr{E}_X(V), \M_X(V) \times X)} A(\M_X(V) \times X) \xrightarrow{\pi_*} A(X).$$
\end{proposition}
\begin{proof}
This is immediate from what may be regarded as the \emph{intended} meaning of the $A$-theory characteristic, but the detailed identification is not direct by definition, because the \emph{actual} definition of the $A$-theory characteristic of a fibration $p$ makes use of the simplex category of the base in order to `straighten' the fibration $p$ and produce a model for its classifying map. In the case of $p_X^V$, however, such a model for the classifying map 
is simply the inclusion map $\M_X(V) \to |w \Rfd(X)|$, and the total space $\mathscr{E}_X(V)$ is the homotopy colimit 
of the forgetful functor from $w \C_X(V)$ to the category of spaces, given on objects by $(V', X) \mapsto V'$. The composite map displayed 
in the proposition is obtained from the following map: 
\begin{equation} 
\label{identification-of-inclusion-map} |\simp(N_\bullet w \C_X(V))| \to |w \Rfd(\M_X(V) \times X)| \xrightarrow{\pi_*} |w \Rfd(X)|, 
\end{equation}
where the first map arises from the image of the coassembly functor at the object $(\mathscr{E}_X(V), \M_X(V) \times X)$. Then the inclusion map $\M_X(V) \to |w \Rfd(X)|$ is identified canonically up to homotopy with the composition \eqref{identification-of-inclusion-map}, using the last-vertex functor
$$\simp(N_{\bullet} w \C_X(V)) \to w \C_X(V), \ (\sigma \colon [m] \to w \C_X(V)) \mapsto (\sigma(m) \in w\C_X(V)).$$ 
\end{proof}

\section{$A$-theory transfer maps and the Waldhausen trace}

\subsection{Recollections} Let $p \colon E \to B$ be a fibration with homotopy finite fibers. The transfer map in $A$-theory, 
\begin{equation} \label{transfer-map-A-theory} 
p^* \colon A(B) \to A(E),
\end{equation} is induced by the exact functor that is defined by pulling back retractive spaces
$$
 p^* \colon \Rfd(B) \to \Rfd(E), \ (V, B) \mapsto (V \times_B E, E).
$$
Note that the functor $p^*$ is well-defined, in particular, it preserves the finiteness condition, and that it is clearly functorial with respect to compositions of fibrations. 
The transfer map \eqref{transfer-map-A-theory} is also identified with the map 
$$A(B) \cong A(B) \wedge S^0 \xrightarrow{\mathrm{Id} \wedge (E \times S^0, E)_+} A(B) \wedge A(p) \to A(E)$$
where the last map is the product pairing in bivariant $A$-theory \cite[Appendix A]{RS1}. The $A$-theory transfer map is also natural with respect to pullbacks in the following 
sense: given a pullback square 
$$
\xymatrix{
E' \ar[d]_{p'} \ar[r]^{g'} & E \ar[d]^p \\
B' \ar[r]_g & B
}
$$ 
where $p$ is a fibration with homotopy finite fibers, then the square 
$$
\xymatrix{
A(E') \ar[r]^{g'_*} & A(E) \\
A(B') \ar[u]_{p'^*} \ar[r]_{g_*} & A(B)\ar[u]_{p^*}
}
$$ 
commutes (up to a canonical homotopy). 

\smallskip

\noindent There is a natural map of infinite loop spaces, called the \emph{unit map},
$$Q(X_+) \colon = \Omega^{\infty} \Sigma^{\infty} X_+ \xrightarrow{\eta_X} A(X),$$ 
whose cofiber is denoted by $\Wh(X)$. The unit map admits several equivalent descriptions: 
\begin{itemize} 
\item[(a)] It is given by the composition of maps of spectra
$$\mathbb{S} \wedge X_+ \xrightarrow{\eta \wedge \mathrm{Id}_+} \mathbf{A}(*) \wedge X_+ \xrightarrow{\alpha} \mathbf{A}(X)$$
where $\eta \colon \mathbb{S} \to \mathbf{A}(*)$ is the unit of the ring spectrum $\mathbf{A}(*)$, defined by $[S^0] \in A(*)$, and $\alpha$ is the assembly map for $A$-theory in the sense of \cite{WW}.
\item[(b)] It can be modeled by the inclusion of the full subcategory of $\Rfd(X)$ given by the retractive spaces which are (homotopically) discrete relative to $X$ (see, for example, \cite{RS1}).
\item[(c)] It can be defined using the models for stable homotopy and $A$-theory from Waldhausen's manifold approach \cite{Wa2}.
\item[(d)] It can be identified with a natural map from the cobordism category of manifolds with boundary to $A$-theory \cite{RS1, RS2, RS3}. 
\end{itemize} 
The relationship between the unit map and the $A$-theory characteristic is explained in the following proposition. 

\begin{proposition} \label{chi-eta-transfer}
Let $p \colon E \to B = |B_{\bullet}|$ be a fibration with homotopy finite fibers. Then the composite map 
$$B \subset QB_+ \xrightarrow{\eta_B} A(B) \xrightarrow{p^*} A(E)$$
is naturally homotopic to $\chi(p)$. 
\end{proposition}
\begin{proof}
Using the model for the unit map constructed in \cite[3.4]{RS1}, the composite map $B \subset QB_+ \xrightarrow{\eta_B} A(B)$ is identified with the map induced by the functor:
$$\simp(B) \to w \Rfd(B), \ (\sigma \colon \Delta^n_{\bullet} \to B_{\bullet}) \mapsto (B \hookrightarrow B \sqcup \Delta^n \xrightarrow{\mathrm{Id} \sqcup |\sigma|} B),$$
followed by the canonical map $|w \Rfd(B)| \to A(B)$. (More precisely, the comparison with \cite[3.4]{RS1} also uses a canonical identification $|\simp(B)| \simeq |\mathrm{sing}_{\bullet}B|$ where $\mathrm{sing}_{\bullet}$ denotes the singular set functor.) 
Then the identification with $\chi(p)$ follows easily from the definitions. 
\end{proof} 

\begin{remark}
A more general version of Proposition \ref{chi-eta-transfer} is as follows: suppose that $(V, B) \in \Rfd(B \xrightarrow{=} B)$ and let $\chi(V, B) \colon B \to A(B)$ be the $A$-theory characteristic of $(V, B)$. 
Then the composition $B \xrightarrow{\chi(V, B)} A(B) \xrightarrow{p^*} A(E)$ agrees with the $A$-theory characteristic $\chi(V', E) \colon B \to A(E)$ of the pullback $(V', E) \in \Rfd(p)$ of $(V, B)$ along $p \colon E \to B$. Proposition \ref{chi-eta-transfer} is essentially the special case for $(V, B) = (B \times S^0, B)$.  
\end{remark}

The unit map $\eta_X$ admits a natural splitting map (as infinite loop spaces) given by the \emph{Waldhausen trace map} \cite{Wa3}: 
$$\Tr_X \colon A(X) \to Q(X_+).$$
As a consequence, the transfer map \eqref{transfer-map-A-theory}, which is also an infinite loop map, may be described in terms of its restrictions to the different summands of the product decomposition of infinite loop spaces: 
$$A(X) \simeq QX_+ \times \Wh(X).$$
The following theorem of \cite{LM} identifies the restriction of the $A$-theory transfer map on the $Q$-summands in terms of the Becker-Gottlieb transfer map. Special cases of this result were also shown in \cite{DJ, D} 
using different methods. We refer to \cite{BG, DWW, KM} for the definition and properties of the Becker-Gottlieb transfer maps.

\begin{theorem} [Lind--Malkiewich \cite{LM}] \label{Lind-Malkiewich}
Let $p \colon E \to B$ be a fibration with homotopy finite fibers. Then the composite map 
$$QB_+ \xrightarrow{\eta_B} A(B) \xrightarrow{p^*} A(E) \xrightarrow{\Tr_E} QE_+$$
is naturally homotopic to the Becker-Gottlieb transfer map $\tr_p$ of $p$.
\end{theorem}
\begin{proof}
See \cite[Corollary 1.3]{LM}. We note that this is shown more generally for fibrations $p$ for which each fiber 
is a retract up to homotopy of a finite CW complex. 
\end{proof}

\begin{remark}
We emphasize that the composite $QB_+ \xrightarrow{\eta_B} A(B) \xrightarrow{p^*} A(E)$ does \emph{not} factor through the unit map $\eta_E \colon QE_+ \to A(E)$ in general. According to the Dwyer-Weiss-Williams smooth index theorem, this holds in the case where $p$ is fiberwise equivalent to a fiber bundle of compact smooth manifolds \cite{DWW, Wi, RS2}. In the general case, the \emph{obstruction} to such a factorization, that is, the map 
$$QB_+ \xrightarrow{\eta_B} A(B) \xrightarrow{p^*} A(E) \to \Wh(E),$$
encodes higher Whitehead torsion invariants of $p$.
\end{remark}

As a consequence of the results so far, we obtain an identification of the restriction of the Waldhausen trace map $\Tr_X$ along $j_X^V \colon \M_X(V) \to A(X)$ 
in terms of Becker-Gottlieb transfer maps. We recall from Subsection \ref{classifying-1} that for each $(V, X) \in \Rfd(X)$, there is a retractive object in 
$\Rfd(\M_X(V) \times X \xrightarrow{\pr} \M_X(V))$:
$$
\xymatrix{
\M_X(V) \times X  \ar@{>->}[r]_(.6)i \ar[rd]_{\pr} \ar@/^1pc/[rr]^{\mathrm{Id}} & \mathscr{E}_X(V) \ar[r]_(.4)r \ar[d]^{p_X^V} & \M_X(V) \times X \ar[dl]^{\pr} \\
& \M_X(V)
}
$$
which classifies fibrations with fibers weakly equivalent to $(V, X)$. 

\begin{corollary} \label{identify-trace-map}
Let $X$ be a path-connected and homotopy finite space and let $(V, X) \in \Rfd(X)$. Then the map 
$$\M_X(V) \xrightarrow{j_X^V} A(X) \xrightarrow{\Tr_X} QX_+$$ 
is homotopic, naturally in $X$, to the map 
$$\big(\M_X(V) \subset Q(\M_X(V)_+) \xrightarrow{\tr_{p_X^V}} Q(\mathscr{E}_X(V)_+) \xrightarrow{Q(\pi \circ r)_+} QX_+\big) - \mathrm{const}_{\chi(X)}$$
where $\mathrm{const}_{\chi(X)}$ denotes the constant map at $\chi(X)$, that is, the point in $QX_+$ that corresponds to the image of $(X \times S^0, X) \in A(X)$ under $\Tr_X \colon A(X) \to QX_+$.
\end{corollary}
\begin{proof}
This is a direct application of, in order, the identifications in Proposition \ref{comparison-1}, Proposition \ref{relative-characteristic}, Proposition \ref{chi-eta-transfer}, and Theorem \ref{Lind-Malkiewich}, together with the naturality of $\Tr$. 

\noindent (By Theorem \ref{Lind-Malkiewich} or otherwise, $\chi(X)$ can be also identified with the Becker-Gottlieb transfer map of $X \to *$. The justification for the notation is that the class of $\chi(X)$ in 
$\pi_0 QX_+ \cong \mathbb{Z}$ corresponds to the classical Euler characteristic of $X$ -- this is where 
we use the assumption that $X$ is path-connected.)
\end{proof}

\subsection{Comparison of transfer maps on $\M_X(V)$} Let $p \colon E \to B$ be a fibration with homotopy finite fibers. In this subsection, we will show that the $A$-theory transfer map and the Becker-Gottlieb transfer map 
of $p \colon E \to B$ are compatible along the Waldhausen trace map when we restrict along $j_B^V \colon \M_B(V) \to A(B)$ for any $(V, B) \in \Rfd(B)$. 

The proof will make essential use of the following result.

\begin{theorem} [Klein--Malkiewich \cite{KM}] \label{BG-functoriality}
Let $q\colon V \to E$ and $p \colon E \to B$ be fibrations with homotopy finite fibers. Then $\tr_{p\circ q}\colon QB_+ \to QV_+$ is homotopic to the composition $QB_+ \xrightarrow{\tr_p} QE_+  \xrightarrow{\tr_p} QV_+$ (as maps of infinite loop spaces).  
\end{theorem}  
\begin{proof}
See \cite[Theorem A]{KM}. We note that this is shown more generally for fibrations for which each fiber is a retract up to homotopy of a finite CW complex. 
\end{proof}

\begin{theorem} \label{main-thm-1}
Let $p \colon E \to B$ be a fibration with homotopy finite fibers where $B$ is homotopy finite. Let $(V, B) \in \Rfd(B)$. Then the following diagram commutes up to homotopy: 
\begin{equation} \label{pre-group-completion}
\xymatrix{
\M_B(V) \ar[r]^{j_B^V} \ar[rd]_{p^*j_B^V} & A(B) \ar[r]^{\Tr_B} & QB_+ \ar[d]^{\tr_p} \\
& A(E) \ar[r]_{\Tr_E} & QE_+.
}
\end{equation}
\end{theorem}
\begin{proof}
We denote $\M \colon = \M_B(V)$ and $\mathscr{E} \colon = \mathscr{E}_B(V)$ for notational simplicity. Since both $X \mapsto A(X)$ and $X \mapsto QX_+$ send finite coproducts to finite products 
up to homotopy equivalence, we may assume that $E$ and $B$ are path-connected by restricting to path components if necessary. Note that $E$ is also homotopy finite. 

Firstly, we analyse the lower composition of \eqref{pre-group-completion} using our previous identifications. By Proposition \ref{comparison-1} and the 
naturality properties of the $A$-theory transfer map and of the trace map $\Tr$, the lower composition agrees up to homotopy with the upper composition 
in the following homotopy commutative diagram:
\begin{equation} \label{diagram-1}
\xymatrix{
\M \ar[rr]^(.4){\chi(\mathscr{E}, \M \times B)} \ar[rrd]_{j_B^V} && \ A(\M \times B) \ar[r]^{(\mathrm{Id} \times p)^*} \ar[d]_{\pi_*} & A(\M \times E) \ar[d]_{\pi_*} \ar[r]_{\Tr_{\M \times E}} & Q(\M \times E)_+ \ar[d]^{Q\pi_+} \\
&& A(B) \ar[r]_{p^*} & A(E) \ar[r]_{\Tr_E} & QE_+.
}
\end{equation}
Consider the fibration $p'$ which is the pullback of the fibration $\mathrm{Id} \times p$ along the retraction map $r\colon \mathscr{E} \to \M \times B$,
\[
 \xymatrix{
 \mathscr{P} \ar[r]^(.4){r'} \ar[d]_{p'} & \M \times E \ar[d]^{\mathrm{Id} \times p} \\
 \mathscr{E} \ar[r]^(.4)r \ar[d]_{p_B^V} & \M \times B \ar[dl]^{\pr} \\
\M.
 }
\]
By the definition of the $A$-theory characteristic (cf. Proposition \ref{chi-eta-transfer}), the upper composition in the previous diagram \eqref{diagram-1} agrees up to homotopy with the composite map
$$\M \xrightarrow{\chi(\mathscr{P}, \M \times E)} A(\M \times E) \xrightarrow{\Tr_{\M \times E}} Q(\M \times E)_+ \xrightarrow{Q\pi_+} QE_+.$$
The last composite map can be identified in terms of Becker-Gottlieb transfer maps using our previous identifications. More specifically, using Proposition \ref{relative-characteristic}, Proposition \ref{chi-eta-transfer} and Theorem \ref{Lind-Malkiewich}, the last composite map is homotopic to the difference 
of homotopy classes of maps:
\begin{equation} \label{difference-1} 
[\M\subset Q \M_+ \xrightarrow{\tr_{p_B^V  p'}} Q\mathscr{P}_+ \xrightarrow{Q(\pi r')_+} QE_+] - [\M\subset Q \M_+ \stackrel{\tr_{\pr}}{\to} Q(\M \times E)_+ \stackrel{Q\pi_+}{\to} QE_+]
\end{equation}
where the second class is the class of the constant map $\mathrm{const}_{\chi(E)}$. On the other hand, by Corollary \ref{identify-trace-map} and the naturality properties of the 
Becker-Gottlieb transfer map $\tr$ with respect to pullback squares, the upper composition in \eqref{pre-group-completion} can be identified up to homotopy with the difference of homotopy classes of maps: 
\begin{equation} \label{difference-2} 
\begin{split}
[\M \subset Q\M_+ \xrightarrow{\tr_{p_B^V}} Q\mathscr{E}_+ \xrightarrow{Q(\pi  r)_+} QB_+ \xrightarrow{\tr_p} 
Q E_+] - [\M \xrightarrow{\mathrm{const}_{\chi(B)}} Q B_+ \xrightarrow{\tr_p} QE_+] =\\
[\M \subset Q \M_+ \xrightarrow{\tr_{p_B^V}} Q\mathscr{E}_+ \xrightarrow{\tr_{p'}} Q\mathscr{P}_+ \xrightarrow{Q(\pi r')_+} 
QE_+] - \mathrm{const}_{\chi(F) \chi(B) = \chi(E)}
\end{split}
\end{equation}
where $F$ denotes the fiber of $p$. The last two homotopy classes \eqref{difference-1} and \eqref{difference-2} agree by Theorem \ref{BG-functoriality}.  
\end{proof}

\section{Comparison of transfer maps}

\subsection{Preliminaries on acyclic maps} A space $X$ is called \emph{acyclic} if $\widetilde{H}(X; \ZZ) = 0$. A map $f \colon X \to Y$ is called \emph{acyclic} if each of its homotopy fibers is an acyclic space. Equivalently, a map $f \colon X \to Y$ is acyclic if it induces isomorphisms 
$H_*(X; f^*A) \cong H_*(Y; A)$ for any local coefficient system of abelian groups $A$ on $Y$. It follows that acyclic maps are closed under homotopy pullbacks and homotopy pushouts. Assuming that $X$ is path-connected, the  acyclic maps $f \colon X \to Y$ are exactly the maps which are given by a plus construction $X \to X^+_P$ with respect to a perfect normal subgroup $P$ of $\pi_1(X, x)$. We refer to \cite{HH, Ra} for further characterizations of acyclic maps and more details about their properties. 

We will need the following result about the relationship between acyclic maps and homotopy finiteness. We recall that a group $G$ is \emph{locally perfect} if every finitely 
generated subgroup of $G$ is contained in a finitely generated perfect subgroup of $G$. 

\begin{theorem}[Hausmann--Vogel \cite{HV}] \label{H-V}
Let $X$ be a path-connected space and let $f \colon X \to X^+_P$ be the plus construction with respect to a locally perfect normal subgroup $P$ of $\pi_1(X, x)$. 
Suppose that $X^+_P$ is homotopy finite. Then there is a finite CW complex $L$ and a map $\gamma \colon L \to X$ such that the composition $L \xrightarrow{\gamma} X \xrightarrow{f} X^+_P$ is again acyclic.  
\end{theorem}
\begin{proof}
This is a special case of \cite[Theorem 3.1]{HV}.  
\end{proof}
 
\subsection{The main result} Our main result below shows that the Waldhausen trace map is natural up to \emph{weak} homotopy with respect to the $A$-theory and 
Becker-Gottlieb transfer maps respectively. We recall that two maps $f, g \colon X \to Y$ are \emph{weakly homotopic} if for any map $h \colon K \to X$
where $K$ is homotopy finite, the maps $fh, gh \colon K \to Y$ are homotopic. In this case, we also say that $f$ and $g$ agree up to weak homotopy. Clearly, it suffices to consider only homotopy finite spaces $K$ which are path-connected. 

Using the definition of $A$-theory in terms of group completion (see Subsections \ref{recollect-1} and \ref{classifying-1}) and the `group completion' theorem \cite{McDS, RW}, there is an acyclic map for any 
based path-connected space $X$:
$$\iota \colon \ZZ \times \M_{X}^{\infty} : = \ZZ \times (\hocolim_{n,k \to \infty} \M_{X}^{n,k}) \longrightarrow \ZZ \times (\M_X^{\infty})^+ \simeq A(X).$$

\begin{theorem} \label{main-thm-2}
Let $p \colon E \to B$ be a fibration with homotopy finite fibers. Then 
\begin{equation} \label{main-square-1}
\xymatrix{
A(B) \ar[d]_{p^*} \ar[r]^{\Tr_B} & QB_+ \ar[d]^{\tr_p} \\
A(E) \ar[r]_{\Tr_E} & QE_+
}
\end{equation}
commutes up to weak homotopy.
\end{theorem}
\begin{proof} 
We may assume without loss of generality that $B = |B_{\bullet}|$ is the geometric realization of a simplicial set. Let $K$ be 
a connected finite CW complex and let $h \colon K \to A(B)$ be a map. We need to show that the following diagram commutes up to homotopy:
\begin{equation} \label{weak-main-square-1}
\xymatrix{
K \ar[r]^{h} \ar[rd]_{p^*h} & A(B) \ar[r]^{\Tr_{B}} & QB_+ \ar[d]^{\tr_{p}} \\
& A(E) \ar[r]_{\Tr_{E}} & QE_+.
}
\end{equation} 
Since $A$-theory preserves homotopy filtered colimits, the map $h$ factors up to homotopy as $K \xrightarrow{h'} A(B') \to A(B)$ where $B' \subset B$ is a finite subcomplex. 
Therefore it suffices to restrict to the case where $B$ is a finite CW complex (and therefore $E$ is homotopy finite) because of the naturality properties of $p^*$ and $\tr_p$ with respect to pullbacks. Since both 
functors $X \mapsto QX_+$ and $X \mapsto A(X)$ send finite coproducts to finite products up to homotopy equivalence, it also suffices to restrict to the case where $E$ and $B$ are path-connected. 

Note that \eqref{main-square-1} commutes at the level of $\pi_0$ because it can be directly identified with the commutative square
(assuming here that $E$ and $B$ are path-connected):
$$
\xymatrix{
\ZZ \ar[d]_{\chi(F)} \ar[r]^{\cong} &  \ZZ \ar[d]^{\chi(F)} \\
\ZZ \ar[r]_{\cong} &  \ZZ
}
$$
where $F$ denotes the fiber of $p$. The identification of the left vertical homomorphism is 
immediate from the definition of the $A$-theory transfer. The identification of the right vertical homomorphism is well known from the properties of Becker-Gottlieb transfer maps (see \cite{BG}). 

As a consequence of these observations, we see that it suffices to consider the case where $B = |B_{\bullet}|$ is a based connected finite CW complex and the map $h \colon K \to A(B)$ factors through $\{0\} \times \M_{B}^{\infty}{}^+ \rightarrow A(B)$ as this is the inclusion of a component up to homotopy equivalence. Then we consider the homotopy pullback: 
$$
\xymatrix{
W \ar[d]_f \ar[r]^{\overline{h}} & \M_{B}^{\infty} \ar[d]^{\iota_0} \\
K \ar[r]_{h} & \M_{B}^{\infty}{}^+ 
}
$$
where $\iota_0$ and (therefore also) $f$ are acyclic maps. Moreover, $\iota_0$ is the plus construction with respect to the perfect normal subgroup $E(\ZZ[\pi_1 B])$, which is also locally perfect (see the proof of \cite[Corollary 4.2]{HV}). 
The map $f$ is then a plus construction with respect to a perfect normal subgroup which is a central extension of $E(\ZZ[\pi_1 B])$, therefore it is also locally perfect 
(see \cite[Proposition 5.5]{Vo}). 

By Theorem \ref{H-V}, there is a finite CW complex $K'$ and a map $\gamma \colon K' \to W$ such that the composition 
$j \colon K' \xrightarrow{\gamma} W \xrightarrow{f} K$ is again acyclic. Since $QE_+$ is an $H$-space, it suffices to show that 
\eqref{weak-main-square-1} commutes up to homotopy after precomposition with the acyclic map $j \colon K' \to K$. Using that $K'$ is finite, 
the map $K' \xrightarrow{\gamma} W \xrightarrow{\overline{h}} \M_{B}^{\infty}$ factors up to homotopy through $\M_B^{n,k} \subset \M_B^{\infty}$ for some $n, k \geq 0$. 
Then the result follows from Theorem \ref{main-thm-1}.  
\end{proof}

\begin{corollary} 
Let $p \colon E \to B$ be a fibration with homotopy finite fibers. Then the composition of infinite loop space maps
$$\Wh(B) \stackrel{\subset}{\to} A(B) \xrightarrow{p^*} A(E) \xrightarrow{\Tr_E} QE_+$$
is weakly homotopic to the trivial map.
\end{corollary} 
 
\section{Final Remarks}

\subsection{Further questions} There are two obvious further questions we can ask about the commutativity of \eqref{main-square-1}: 
\begin{itemize}
\item[(1)] Does \eqref{main-square-1} commute up to homotopy (and not just up to \emph{weak} homotopy)?
\item[(2)] If yes, is there a choice of homotopy for \eqref{main-square-1} which is natural in $p$ (with respect to operations, such as: (i) composition of fibrations, (ii) pullback of fibrations, (iii) homotopy filtered colimits 
of fibrations, etc.)?
\end{itemize}
Even though an answer to Question (1) is not immediate using our approach, we believe that the commutativity of \eqref{main-square-1} holds also up 
to homotopy. For this, it seems sufficient to refine Theorem \ref{main-thm-1} so that the choice of the homotopy in \eqref{pre-group-completion} is natural 
in $(V, B)$ (in a suitable sense). In this case, we could then conclude that 
\eqref{pre-group-completion} commutes up homotopy also in the stabilized case of $\M_B^{\infty}$ in place of $\M_B(V)$. Obtaining such a refinement of Theorem \ref{main-thm-1} 
essentially depends on the naturality properties of the homotopy for the functoriality of the Becker-Gottlieb transfer maps in Theorem \ref{BG-functoriality}. 
We do not know whether the necessary naturality properties hold for the homotopy shown in \cite{KM}. On the other hand,  
we also hope that an altogether different approach to Question (1) might be possible. We point out 
that an affirmative answer to Question (1) implies Theorem \ref{BG-functoriality} -- as a consequence of the naturality of the $A$-theory transfer maps. 

Similarly, assuming that \eqref{main-square-1} commutes up to homotopy, an answer to Question (2) requires further refinements of our results regarding 
the naturality properties of certain identifications. Note that a positive answer to Question (2) also implies Theorem \ref{Lind-Malkiewich}. In addition, 
a positive answer to Question (2) implies analogous naturality properties for the functoriality of the Becker-Gottlieb transfer maps, using again the naturality properties of the $A$-theory transfer maps.  
Thus, an independent approach to Question (2) would be desirable also in order to conclude a refined version of Theorem \ref{BG-functoriality}. 

\subsection{Transfer maps in THH} Given a based path-connected space $X$, we let $\mathrm{THH}(\mathbb{S}[\Omega X])$ denote the topological Hochschild homology of the spherical group ring associated to the space $X$. It is known that there is a 
homotopy equivalence of infinite loop spaces $\mathrm{THH}(\mathbb{S}[\Omega X]) \simeq QLX_+$ where $LX$ denotes the free loop space of $X$. Given a fibration $p \colon E \to B$ with homotopy finite fibers between based 
path-connected spaces $E$ and $B$, there is a transfer map in THH:
$$QLB_+ \simeq \mathrm{THH}(\mathbb{S}[\Omega B]) \xrightarrow{p^*} \mathrm{THH}(\mathbb{S}[\Omega E]) \simeq QLE_+.$$
We refer to \cite{LM} for the definition and properties of the THH (or free loop) transfer map. It is known that the Waldhausen trace map $\Tr$ factors naturally through THH,
$$A(X) \longrightarrow \mathrm{THH}(\mathbb{S}[\Omega X]) \simeq QLX_+ \xrightarrow{\mathrm{ev}} QX_+,$$ 
where the first map is a topological version of the Dennis trace map and the last map is given by evaluation at $1 \in S^1$ (see \cite{LM, Wa3}). Thus, given a fibration $p \colon E \to B$ as above, we may consider the following diagram: 
\begin{equation} \label{THH}
\xymatrix{\ar@{}[dr] | {\checkmark} 
A(B) \ar[d]^{p^*} \ar[r] & \mathrm{THH}(\mathbb{S}[\Omega B]) \ar@{}[dr] | {\text{\sffamily X}}  \ar[d]^{p^*} \ar[r] & QB_+ \ar[d]^{\tr_p} \\
A(E) \ar[r] & \mathrm{THH}(\mathbb{S}[\Omega E]) \ar[r] & QE_+. 
}
\end{equation}
The left--hand square commutes up to canonical homotopy -- because the trace map from algebraic $K$-theory to THH is natural with respect to exact functors. On the other hand, 
the right--hand square does \emph{not} commute by results of \cite{LM, Sc} -- for example, it fails to commute in general for the universal covering of $BG$ 
where $G$ is a finite group. It would be interesting to know if this failure of commutativity for the right--hand square is corrected when one restricts to the homotopy 
$S^1$-fixed points of THH.

\end{document}